    \documentclass[oneside,english]{amsart}
    \usepackage[T1]{fontenc}
    \usepackage[latin9]{inputenc}
    \usepackage{geometry}
    \usepackage{verbatim}
    \usepackage{amstext}
    \usepackage{amsthm}
    \usepackage{hyperref}
    \usepackage{mathtools}
    \usepackage{graphicx}
    \usepackage{float}
    \usepackage{algorithm}
    \usepackage[noend]{algpseudocode}
    \usepackage{lipsum}
    \usepackage{xcolor}
    
    \makeatletter
    \numberwithin{equation}{section}
    \numberwithin{figure}{section}
    \theoremstyle{plain}
    \newtheorem{thm}{\protect\theoremname}
      \theoremstyle{plain}
      \newtheorem{lem}[thm]{\protect\lemmaname}
      \theoremstyle{plain}
      
      \theoremstyle{remark}
      
      \theoremstyle{plain}
      \newtheorem{conj}[thm]{\protect\conjname}

    \DeclareMathOperator{\B}{\mathcal{B}}
    
    \DeclarePairedDelimiter{\p}{(}{)}
    
    \DeclarePairedDelimiter{\floor}{\lfloor}{\rfloor}
    
    \makeatother
    
    \algnewcommand\algorithmicto{\textbf{to}}
    \algnewcommand\To{\algorithmicto{} }
    \algnewcommand\algorithmicswitch{\textbf{switch}}
    \algnewcommand\algorithmiccase{\textbf{case}}
    \algdef{SE}[SWITCH]{Switch}{EndSwitch}[1]{\algorithmicswitch\ #1\ \algorithmicdo}{\algorithmicend\ \algorithmicswitch}%
    \algdef{SE}[CASE]{Case}{EndCase}[1]{\algorithmiccase\ #1}{\algorithmicend\ \algorithmiccase}%
    \algtext*{EndSwitch}%
    \algtext*{EndCase}%
    
    \usepackage{babel}
      \providecommand{\lemmaname}{Lemma}
      \providecommand{\corname}{Corollary}
      \providecommand{\remarkname}{Remark}
      \providecommand{\conjname}{Conjecture}
    \providecommand{\theoremname}{Theorem}

    \newcommand{\blfootnote}[1]{{\renewcommand{\thefootnote}{\roman{footnote}}\footnotetext[0]{#1}}}
    
\begin{document}
    
    \title{Extremal Graphs for a Spectral Inequality on Edge-Disjoint Spanning Trees}
    
    \author{Sebastian M. Cioab\u{a}${}^{1}$, Anthony Ostuni${}^{2 \dag}$, Davin Park${}^{2 \dag}$, Sriya Potluri${}^{2 \dag }$, Tanay Wakhare${}^{3 \ast}$, Wiseley Wong${}^{2}$}
 
    \begin{abstract}
     Liu, Hong, Gu, and Lai proved if the second largest eigenvalue of the adjacency matrix of graph $G$ with minimum degree $\delta \ge 2m+2 \ge 4$ satisfies $\lambda_2(G) < \delta - \frac{2m+1}{\delta+1}$, then $G$ contains at least $m+1$ edge-disjoint spanning trees, which verified a generalization of a conjecture by Cioab\u{a} and Wong. We show this bound is essentially the best possible by constructing $d$-regular graphs $\mathcal{G}_{m,d}$ for all $d \ge 2m+2 \ge 4$ with at most $m$ edge-disjoint spanning trees and $\lambda_2(\mathcal{G}_{m,d}) < d-\frac{2m+1}{d+3}$. As a corollary, we show that a spectral inequality on graph rigidity by Cioab\u{a}, Dewar, and Gu is essentially tight.
    \end{abstract}
    
    \maketitle
    
    
    \blfootnote{${}^{\dag}$ denotes joint first authorship}
    \blfootnote{${}^{\ast}$ denotes corresponding author, \email{twakhare@mit.edu}}
    \blfootnote{$^{1}$~University of Delaware, Newark, DE 19716, USA}
    \blfootnote{$^{2}$~University of Maryland, College Park, MD 20878, USA}
    \blfootnote{$^{3}$~Department of Electrical Engineering and Computer Science, MIT, Cambridge, MA 02139, USA}
    
    
    \section{Introduction}
     Let $G = (V,E)$ be a finite, simple graph on $n$ vertices, and let $\lambda_1 \ge \lambda_2 \ge \cdots \ge \lambda_n$ and $\mu_1 \le \mu_2 \le \cdots \le \mu_n$ be the eigenvalues of its adjacency and Laplacian matrices, respectively. Recall $\mu_i + \lambda_i = d$ for $d$-regular graphs \cite[Ch. 1]{brouwer2011spectra}. Additionally, let $\sigma(G)$ denote the maximum number of edge-disjoint spanning trees in $G$, sometimes referred to as the spanning tree packing number (see Palmer \cite{spanningpackingpalmer} for a survey of this parameter). Motivated by Kirchhoff's celebrated  matrix tree theorem on the number of spanning trees of a graph \cite{kirchhoff1847ueber} and a question of Seymour \cite{seymour}, Cioab\u{a} and Wong \cite{cioabaSpanningTrees} considered the relationship between the eigenvalues of a regular graph and $\sigma(G)$. 
     
     They obtained a result by combining two useful theorems. The Nash-Williams/Tutte theorem \cite{nash1961edge, tutte1961problem} (described in Section \ref{trees}) implies that if $G$ is a $(2m+2)$-edge-connected graph, then $\sigma(G) \ge m+1$. Additionally, Cioab\u{a} \cite{cioabua2010eigenvalues} showed if $G$ is a $d$-regular graph and $r$ is an integer with $2 \le r \le d$ such that $\lambda_2(G) < d - 2(r-1)/(d+1)$, then $G$ is $r$-edge-connected. These facts imply that if $G$ is a $d$-regular graph with $\lambda_2(G) < d - 2(2m+1)/(d+1)$ for some integer $m$, with $2 \le m+1 \le \lfloor d/2 \rfloor$, then $G$ contains $m+1$ edge-disjoint spanning trees. Cioab\u{a} and Wong conjectured the following factor of two improvement, which they verified for $m \in \{1,2\}$. 
     \begin{conj}[\cite{cioabaSpanningTrees}]
     Let $m \ge 1$ be an integer and $G$ be a $d$-regular graph with $d \ge 2m+2$. If $\lambda_2(G) < d - \frac{2m+1}{d+1}$, then $\sigma(G) \ge m+1$.
     \end{conj}
     
     This conjecture attracted much attention, leading to many partial results and generalizations \cite{guthesis, spanningpacking, gu2016edge, hong2016fractional, li2013edge, liu2014edge}. The question was ultimately resolved by Liu, Hong, Gu, and Lai. 
    \begin{thm}[\cite{noteOnTrees2014}] \label{lambda2bound}
    Let $m \ge 1$ be an integer and $G$ be a graph with minimum degree $\delta \ge 2m+2$. If $\lambda_2(G) < \delta - \frac{2m+1}{\delta+1}$, then $\sigma(G) \ge m+1$.
    \end{thm}
    
    We show this bound is essentially the best possible.
    \begin{thm}\label{extremegraph}
    For all $d \ge 2m+2 \ge 4$, the $d$-regular graph $\mathcal{G}_{m,d}$ (defined in Section \ref{construction})
    has at most $m$ edge-disjoint spanning trees 
    and satisfies $$d-\dfrac{2m+1}{d+1}\leq \lambda_2(\mathcal{G}_{m,d})< d-\dfrac{2m+1}{d+3}.$$
    
    \end{thm}
    
    Cioab\u{a} and Wong created special cases of this construction for the families $\mathcal{G}_{1,d}$ and (a slight variant of) $\mathcal{G}_{2,d}$ in \cite{cioabaSpanningTrees} to show that Theorem \ref{lambda2bound} is essentially best possible for $m \in \{1,2\}$. In his PhD thesis \cite{wongthesis}, Wong also constructed the family $\mathcal{G}_{3,d}$ to show that Theorem \ref{lambda2bound} is essentially tight for $m=3$. Based on the family of graphs for the small cases of $m$ that appeared in \cite{cioabaSpanningTrees}, Gu \cite{spanningpacking} constructed a family of multigraphs by replacing every edge with multiple edges to show that the bounds in a multigraph analog of Theorem \ref{lambda2bound} are also the best possible. Additionally, Cioab\u{a}, Dewar, and Gu \cite{cioabua2020spectral} used the variant of $\mathcal{G}_{2,d}$ from \cite{cioabaSpanningTrees} to show that a sufficient spectral condition for graph rigidity is essentially the best possible. We generalize their result in Section \ref{rigidity}.

    In Section \ref{constructionapplication}, we will construct the family of graphs $\mathcal{G}_{m,d}$ and prove the lower bound of Theorem \ref{extremegraph}. In Section \ref{polynomial}, we will explicitly describe the characteristic polynomial of $\mathcal{G}_{m,d}$ (Theorem \ref{charpoly}), and in Section \ref{bound}, we will use the characteristic polynomial to prove the upper eigenvalue bound of Theorem \ref{extremegraph}. The proof of the second eigenvalue bound uses a classical number theoretic technique, Graeffe's method (see Lemma \ref{Graeffe}), which to the best of our knowledge has not previously been used for second eigenvalue bounds. This approach should generalize to upper bounds on the roots of interesting combinatorial polynomials. 

    
    \section{Graph Construction} \label{constructionapplication}
    
    \subsection{Construction}\label{construction}
 We construct a family of graphs $\mathcal{G}_{m,d}$ such that $\lambda_2(\mathcal{G}_{m,d})<d-\frac{2m-1}{d+3}$, but $\sigma(\mathcal{G}_{m,d})\le m$ for all $d\geq 2m+2\geq 4$.  The graph $\mathcal{G}_{m,d}$ contains $2m+1$ copies of $K_{d+1}$, each with a deleted matching of size $m$.  Then $m(2m+1)$ edges are added in a circulant manner to connect the vertices among the $2m+1$ cliques with the deleted matchings.

         Let  $d\geq 2m+2\geq 4$. The vertex set of $G=\mathcal{G}_{m,d}$ consists of all ordered pairs $(i,j)$ where $0\leq i\leq 2m$ and $0\leq j\leq d$. Let $H_i=\{(i,j)\mid 0\leq j\leq d\}$, and let the subgraph induced by $H_i$ be  $G[H_i]= K_{d+1}\setminus E_i$, where 
    $$E_i=\{(i,2a-2)\sim (i,2a-1) \mid 1\le a\le m\}.$$
    
    Now we connect edges among the  $H_i$. Let
    $$E'=\{(i,2j+1)\sim (i+j+1,2j) \mid  0\leq i\leq 2m,\: 0\leq j\leq m-1\},$$
    where $i+j+1$ is taken modulo $2m+1$.
    Then 
    $$E(\mathcal{G}_{m,d}) =\bigcup_{i=0}^{2m}E(G[H_i])\cup E'.$$
 This construction makes $\mathcal{G}_{m,d}$ a connected $d$-regular graph.  See Figures \ref{treepic1} and \ref{treepic2}.
 
 \subsection{Spanning Trees}\label{trees}
    Our result, like many prior results on edge-disjoint spanning trees, crucially relies on a theorem from Nash-Williams and Tutte, which converts a condition on $\sigma(G)$ to one on vertex partitions. If the vertex set $V(G)$ is partitioned into disjoint sets $V_1,\ldots,V_t$, then let $e(V_i,V_j)$ be the number of edges with endpoints in both $V_i$ and $V_j$.
    
    \begin{thm}[Nash-Williams/Tutte \cite{nash1961edge, tutte1961problem}]{\label{nash}} Let $G$ be a connected graph and $k > 0$ be an integer. Then $\sigma(G) \ge k$ if and only if $\sum_{1\leq i < j \leq t}e(V_i,V_j) \ge k(t-1)$ for any partition $V(G) = V_1 \cup \cdots \cup V_t$.
    \end{thm}
    
    As in the previous subsection, let $H_1, \ldots, H_{2m+1}$ be the modified cliques $K_{d+1}$ of $\mathcal{G}_{m,d}$. Since $e(H_i,H_j)=1$, we have
    
    \[\sum_{0\leq i<j\leq 2m}e(H_i,H_j)=\dfrac{2m(2m+1)}{2}<(m+1)(2m).\]
    
    By Theorem \ref{nash}, $\mathcal{G}_{m,d}$ has at most $m$ edge-disjoint spanning trees.  Then Theorem \ref{lambda2bound} implies $\lambda_2(\mathcal{G}_{m,d})\ge d-\frac{2m+1}{d+1}$, yielding the lower bound of Theorem \ref{extremegraph}. It remains to show   $\lambda_2(\mathcal{G}_{m,d})<d-\frac{2m+1}{d+3}$. 

    \begin{figure}[H]
        \centering
        \includegraphics[scale=0.75]{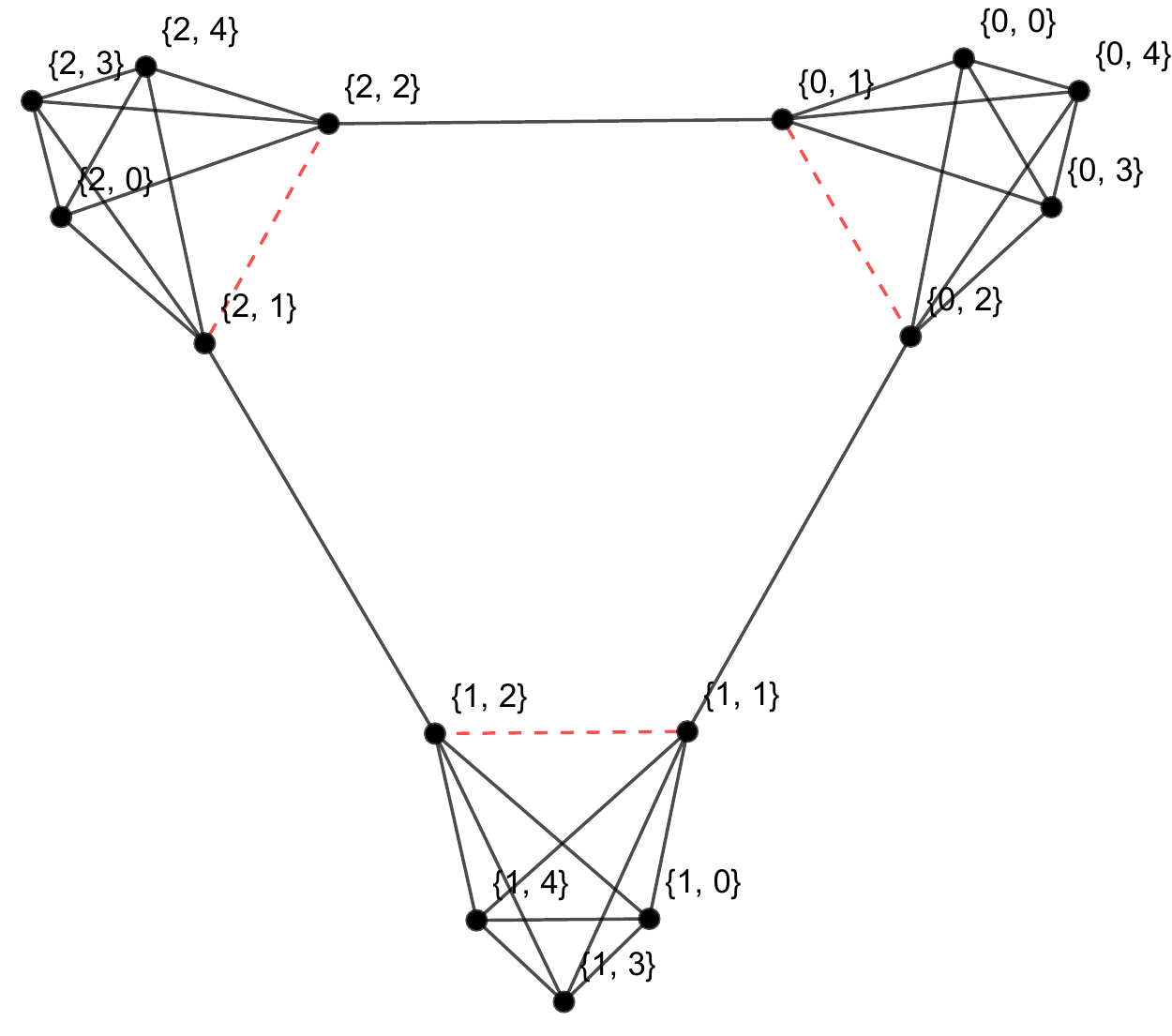}
        \caption{$\mathcal{G}_{1,4}$ with deleted dashed edges and labeled vertices.}
        {\label{treepic1}}
    \end{figure}

    
    \begin{figure}[H]
        \centering
        \includegraphics[scale=0.75]{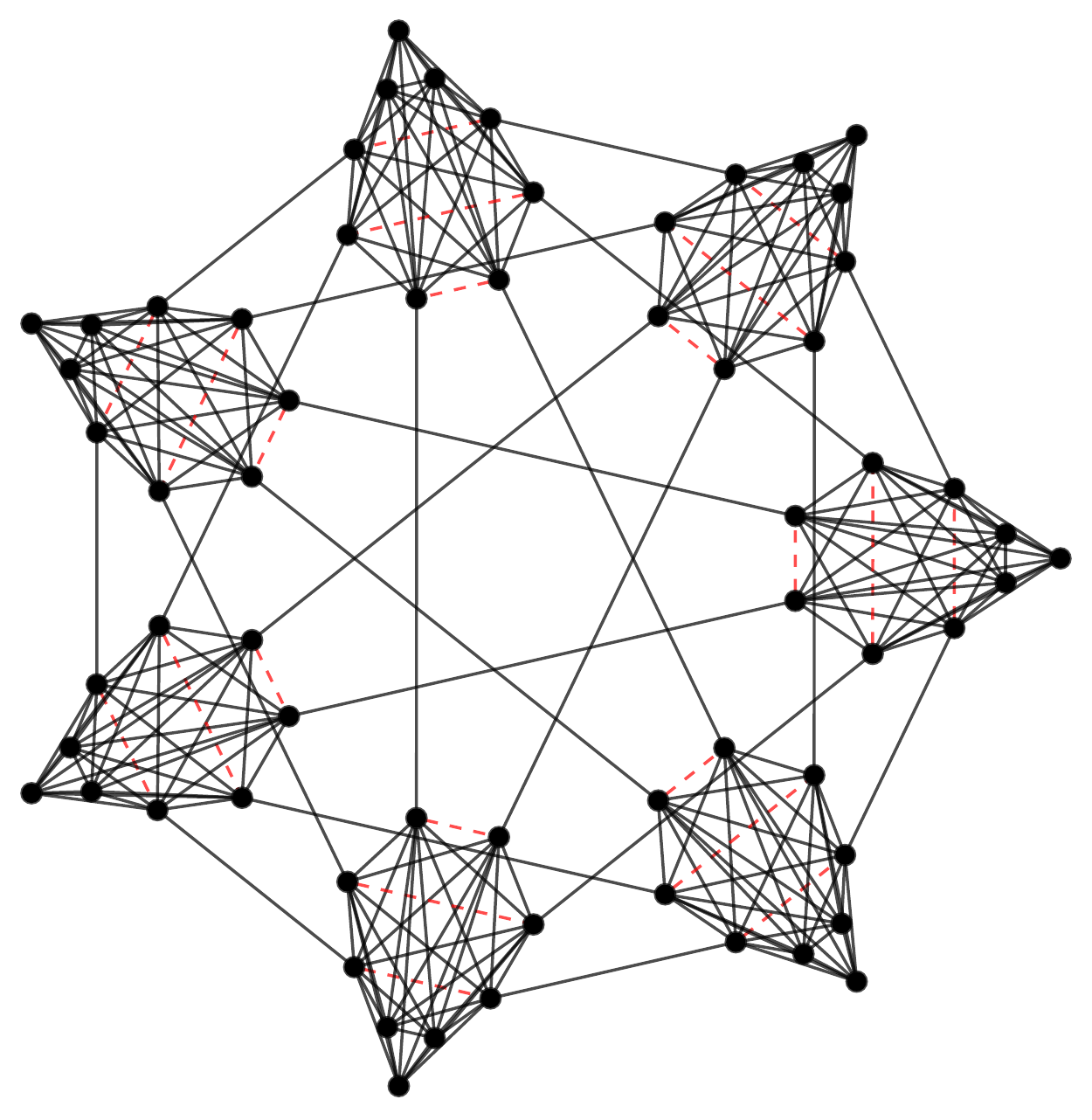}
        \caption{$\mathcal{G}_{3,8}$ with deleted dashed edges.}
        {\label{treepic2}}
    \end{figure}
    
    \section{Characteristic Polynomial} \label{polynomial}
    
The  adjacency matrix of $\mathcal{G}_{m,d}$ is a block circulant matrix. Following \cite{blockCirculantEvalues}, define $\B(\mathbf{b}_0,\mathbf{b}_1,\ldots,\mathbf{b}_{n-1})$ to be the block circulant matrix
    \[\begin{bmatrix}
        \mathbf{b}_0 & \mathbf{b}_1 & \mathbf{b}_2 & \ddots & \mathbf{b}_{n-2} & \mathbf{b}_{n-1}\\
        \mathbf{b}_{n-1} & \mathbf{b}_0 & \mathbf{b}_1 & \ddots & \mathbf{b}_{n-3} & \mathbf{b}_{n-2}\\
        \mathbf{b}_{n-2} & \mathbf{b}_{n-1} & \mathbf{b}_0 & \ddots & \mathbf{b}_{n-4} & \mathbf{b}_{n-3}\\
        \ddots & \ddots & \ddots & \ddots & \ddots & \ddots\\
        \mathbf{b}_2 & \mathbf{b}_3 & \mathbf{b}_4 & \ddots & \mathbf{b}_0 & \mathbf{b}_1\\
        \mathbf{b}_1 & \mathbf{b}_2 & \mathbf{b}_3 & \ddots & \mathbf{b}_{n-1} & \mathbf{b}_0
    \end{bmatrix},\]
    where each $\mathbf{b}_i$ is a square matrix of equal dimension.
    
    \begin{lem}[\cite{blockCirculantEvalues}]\label{blocklemma}
    The characteristic polynomial of a real, symmetric, block circulant matrix $\B(\mathbf{b}_0,\mathbf{b}_1,\ldots,\mathbf{b}_{n-1})$ is given by $$\det(xI - \B(\mathbf{b}_0,\mathbf{b}_1,\ldots,\mathbf{b}_{n-1})) = \prod_{\zeta^n=1} \det(xI-\mathbf{H}_\zeta),$$
        where
    \[\mathbf{H}_\zeta=\mathbf{b}_0+\zeta \mathbf{b}_1 + \cdots + \zeta^{n-1}\mathbf{b}_{n-1},\]
    and $\zeta$ runs over the $n$th roots of unity (including $1$).
    \end{lem}
    
    To determine the characteristic polynomial of $\mathcal{G}_{m,d}$, we will also need the following lemmas from linear algebra.     Let $\mathbf{I}_n$ and $\mathbf{J}_n$ denote the identity matrix and all ones matrix of dimension $n$, respectively.
    
    \begin{lem}\label{matrix_det_lemma}
        Let $\mathbf{A}$ be an invertible matrix and $\mathbf{u},\mathbf{v}$ be column vectors. Then,
        \[\det(\mathbf{A}+\mathbf{uv}^T)=(1+\mathbf{v}^T\mathbf{A}^{-1}\mathbf{u})\det(\mathbf{A}).\]
    \end{lem}
    
    \begin{lem}\label{ai+bj}
   We have
    \begin{align*}
        \det(a\mathbf{I}_n+b\mathbf{J}_n)&=a^n+na^{n-1}b,\\
        (a\mathbf{I}_n+b\mathbf{J}_n)^{-1}&=\frac{1}{a}\mathbf{I}_n-\frac{b}{a(a+nb)}\mathbf{J}_n.
    \end{align*}
    \end{lem}
    
    Finally, the characteristic polynomial of $\mathcal{G}_{m,d}$ will require defining the \textit{Chebyshev polynomials of the first kind} $T_n(z)$ and \textit{Chebyshev polynomials of the second kind} $U_n(z)$. We have \cite[p. 775, Equations (22.3.6) and (22.3.7)]{Abramowitz}
    \begin{align}
     T_n(z) &= \frac{n}{2}\sum_{k=0}^{\floor{n/2} } \frac{(-1)^k}{n-k}\binom{n-k}{k}(2z)^{n-2k}, \label{chebytsum}\\
     U_n(z) &= \sum_{k=0}^{\floor{n/2} } {(-1)^k}\binom{n-k}{k}(2z)^{n-2k}.\label{chebyusum}
    \end{align}
    They are also given by the implicit equations $T_n(\cos \theta) = \cos \p*{n \theta}$ and $U_n(\cos \theta)\sin \theta = \sin ((n+1)\theta)$.
    We prove a few lemmas on the Chebyshev polynomials and their connection to roots of unity.
    
    \begin{lem}\label{chebyprod}
    For $m\geq 1$,  $$ \frac{T_{2m+1}(z)}{z} = 2^{2m}\prod_{j=1}^{m} \left(z^2 - \sin^2 \left( \dfrac{\pi j}{2m+1}  \right)\right).$$
    \end{lem}
    \begin{proof}
    Abramowitz and Stegun \cite[Page 787, Equation 22.16.4]{Abramowitz} provide the zeroes of $T_n(z)$ as $\cos\left( \frac{2\ell-1}{2n}\pi\right), 1\leq 
    \ell\leq n$. The leading coefficient of $T_{2m+1}(z)$ is $2^{2m} $, as can be seen from equation \eqref{chebytsum}, giving the product representation
    \begin{align*}
        T_{2m+1}(z) &= 2^{2m} \prod_{\ell=1}^{2m+1} \left( z - \cos\left(   \frac{\ell-1/2}{2m+1}\pi \right) \right)  \\
        &=2^{2m} z\prod_{\ell=1}^{m} \left( z - \cos\left(   \frac{\ell-1/2}{2m+1}\pi \right) \right)\left( z + \cos\left(   \frac{\ell-1/2}{2m+1}\pi \right) \right)  \\
        &=2^{2m} z\prod_{\ell=1}^{m} \left( z^2 - \sin^2\left(   \frac{\ell-m-1}{2m+1}\pi \right) \right)  \\
        &=2^{2m} z\prod_{\ell=1}^{m} \left( z^2 - \sin^2\left(   \frac{\ell}{2m+1}\pi \right) \right).
    \end{align*}
    
    \end{proof}
    
    \begin{lem}\label{primitive}
    For $m\geq 1$ and any $(2m+1)$-th primitive root of unity $\zeta$,
    \begin{align*}
        \prod_{j=1}^{m} (x^2+2x-1+\zeta^j+\overline{\zeta^j})&=\dfrac{T_{2m+1}(z)}{z},\\
        \sum_{j=1}^{m} \dfrac{2x+\zeta^j+\overline{\zeta^j}}{x^2+2x-1+\zeta^j+\overline{\zeta^j}}&=-\dfrac{1}{2z}+\dfrac{2m+1}{2T_{2m+1}(z)}(T_{2m+1}(z)-(z-1)U_{2m}(z)),
    \end{align*}
    with the change of variables $z=\frac{x+1}{2}$.
    \end{lem}
    
    \begin{proof}
    Applying Lemma \ref{chebyprod},
    \begin{align*}
    \prod_{j=1}^{m} (x^2+2x-1+\zeta^j+\overline{\zeta^j}) &=   2^{2m}  \prod_{j=1}^{m} \left( \left(\frac{x+1}{2}\right)^2-\left( \frac{ \zeta^{j/2}-{\overline{\zeta^{j/2}}}}{2\imath} \right)^2\right) \\
    &= 2^{2m} \prod_{j=1}^{m} \left( \left(\frac{x+1}{2}\right)^2-\sin^2 \left( \frac{\pi j}{2m+1}  \right) \right)  \\
    &=   \frac{ T_{2m+1}\left( \frac{x+1}{2} \right)}{ (x+1)/2 }.
    \end{align*}
    Also, we have
    \begin{align}
    \sum_{j=1}^{m} \dfrac{2x+\zeta^j+\overline{\zeta^j}}{x^2+2x-1+\zeta^j+\overline{\zeta^j}}&=\sum_{j=1}^{m} \frac{\frac{x+1}{2} - \left( \frac{\zeta^{j/2} -\overline{\zeta^{j/2}}}{2\imath } \right)^2
    }{\left(\frac{x+1}{2}\right)^2 - \left( \frac{\zeta^{j/2} -\overline{\zeta^{j/2}}}{2\imath } \right)^2
    }\nonumber\\
    &=\sum_{j=1}^{m} \frac{\frac{x+1}{2} - \sin^2 \left(\frac{\pi j}{2m+1}\right)}{\left(\frac{x+1}{2}\right)^2 - \sin^2 \left(\frac{\pi j}{2m+1}\right)}\label{eq:sinderiv}.
    \end{align}
    Taking a logarithm in Lemma \ref{chebyprod} and then differentiating gives
    \begin{align}
        \log T_{2m+1}(z)-\log z &= \log 2^{2m}+ \sum_{j=1}^{m}\log \left(z^2 -\sin^2 \left( \dfrac{\pi j}{2m+1}  \right) \right),  \nonumber\\
        \dfrac{(2m+1)U_{2m}(z)}{T_{2m+1}(z)} -\frac{1}{z} &= 2z\sum_{j=1}^{m} \frac{1}{ z^2 - \sin^2 \left( \frac{\pi j}{2m+1}  \right)},\label{cheby3}
    \end{align}
    where we used $T'_{n}(z) = nU_{n-1}(z)$ for $n \geq 1$.  This can be verified from the series expansions of both. Substituting equation \eqref{cheby3} into equation \eqref{eq:sinderiv} with $z=(x+1)/2$ gives
    \begin{align*}
    \sum_{j=1}^{m} \frac{z- \sin^2 \left(\frac{\pi j}{2m+1}\right)}{z^2 - \sin^2 \left(\frac{\pi j}{2m+1}\right)} &= \sum_{j=1}^{m} \frac{z^2- \sin^2 \left(\frac{\pi j}{2m+1}\right)}{z^2 - \sin^2 \left(\frac{\pi j}{2m+1}\right)}+\sum_{j=1}^{m} \frac{z-z^2}{z^2 - \sin^2 \left(\frac{\pi j}{2m+1}\right)}\\
    &= m+ (z-z^2)\sum_{j=1}^{m} \frac{1}{z^2 - \sin^2 \left(\frac{\pi j}{2m+1}\right)} \\
    &=m+ \frac{(1-z)(2m+1)}{2} \frac{U_{2m(z)}}{T_{2m+1}(z)} - \frac{1-z}{2z} \\
    &= -\frac{1}{2z} + \frac{2m+1}{2T_{2m+1}(z)} \p*{T_{2m+1}(z)-(z-1)U_{2m}(z)}.
    \end{align*}
    \end{proof}

    We generalize Lemma \ref{primitive} for all roots of unity $\zeta$.
    \begin{lem}\label{nonprimitive}
    For $m\geq 1$ and $2m+1 \ge t \ge 1$, let $\zeta = e^{2\pi \imath t/(2m+1)}$ be a $(2m+1)$-th root of unity, where $\imath=\sqrt{-1}$. Define $g=\gcd(2m+1,t)$ and $n=\frac{2m+1}{g}$. Then    \begin{align*}
        \prod_{j=1}^{m} (x^2+2x-1+\zeta^j+\overline{\zeta^j})&=\dfrac{(2T_n(z))^g}{2z},\\
        \sum_{j=1}^{m} \dfrac{2x+\zeta^j+\overline{\zeta^j}}{x^2+2x-1+\zeta^j+\overline{\zeta^j}}&=-\dfrac{1}{2z}+\dfrac{2m+1}{2T_{n}(z)}(T_{n}(z)-(z-1)U_{n-1}(z)),
        \end{align*}
        with the change of variables $z=\frac{x+1}{2}$.
    \end{lem}
    \begin{proof}
    
    Note that $\zeta$ is a primitive $n$th root of unity. By Lemma \ref{primitive},
    \begin{align*}
        \prod_{j=1}^{\frac{n-1}{2}} (x^2+2x-1+\zeta^j+\overline{\zeta^j})&=\dfrac{T_{n}(z)}{z},\\
        \sum_{j=1}^{\frac{n-1}{2}} \dfrac{2x+\zeta^j+\overline{\zeta^j}}{x^2+2x-1+\zeta^j+\overline{\zeta^j}}&=-\dfrac{1}{2z}+\dfrac{2m+1}{2T_{n}(z)}(T_{n}(z)-(z-1)U_{n-1}(z)).
    \end{align*}
    Notice $m=\frac{g-1}{2}n+\frac{n-1}{2}$. Using the periodicity of roots of unity several times,
    \begin{align*}
        \prod_{j=1}^{m} (x^2+2x-1+\zeta^j+\overline{\zeta^j})&=\p*{\prod_{j=1}^{n} (x^2+2x-1+\zeta^j+\overline{\zeta^j})}^{\frac{g-1}{2}}\p*{\prod_{j=1}^{\frac{n-1}{2}} (x^2+2x-1+\zeta^j+\overline{\zeta^j})}\\
        &=\p*{\prod_{j=1}^{\frac{n-1}{2}} (x^2+2x-1+\zeta^j+\overline{\zeta^j})}^{g}(x^2+2x+1)^{\frac{g-1}{2}}\\
        &=\p*{\dfrac{T_n(z)}{z}}^g(x+1)^{g-1}\\
        &=\dfrac{(2T_n(z))^g}{2z}.
    \end{align*}
    Similarly,
    \begin{align*}
        \sum_{j=1}^{m} \dfrac{2x+\zeta^j+\overline{\zeta^j}}{x^2+2x-1+\zeta^j+\overline{\zeta^j}}&=g\sum_{j=1}^{\frac{n-1}{2}} \dfrac{2x+\zeta^j+\overline{\zeta^j}}{x^2+2x-1+\zeta^j+\overline{\zeta^j}} + \dfrac{g-1}{2}\dfrac{2x+2}{x^2+2x+1}\\
        &=g\p*{-\dfrac{1}{2z}+\dfrac{n}{2T_{n}(z)}(T_{n}(z)-(z-1)U_{n-1}(z))}+\dfrac{g-1}{x+1}\\
        &=-\dfrac{1}{2z}+\dfrac{2m+1}{2T_{n}(z)}(T_{n}(z)-(z-1)U_{n-1}(z)).
    \end{align*}
    \end{proof}
    
    Finally, we compute the characteristic polynomial of $\mathcal{G}_{m,d}$. Denote $\mathbf{J}_{a,b}$ to be the $a\times b$ all ones matrix.

    \begin{thm}\label{charpoly}
    For $d\geq 2m+2\geq 4$, the characteristic polynomial $p(x)$ of $\mathcal{G}_{m,d}$ is
    \[(x+1)^{(d-2m)(2m+1)}\prod_{j=1}^{2m+1} (2T_{n} (z))^{g-1}\left[\p*{-2m+1+\dfrac{2m-d}{z}} T_{n}(z) + (2m+1)(z-1)U_{n-1}(z)\right],\]
    with the change of variables $g=\gcd(2m+1,j)$, $n=\frac{2m+1}{g}$, and $z=\frac{x+1}{2}$.
    \end{thm}
    
    \begin{proof}
    By the symmetric construction of $\mathcal{G}_{m,d}$, its adjacency matrix is a block circulant matrix $\mathcal{B}(\mathbf{b}_0,\ldots,\mathbf{b}_{2m})$. For any $i\neq j$, there is only one edge between $H_i$ and $H_j$. Then for $1\leq i\leq m$, $\mathbf{b}_i$ is a matrix with a single entry of $1$ in position $(2i,2i-1)$. Moreover, $\mathbf{b}_i=\mathbf{b}_{2m+1-i}^T$ for $m+1\leq i\leq 2m$. Finally, $\mathbf{b}_0$ is the adjacency matrix of any $H_i$. That is,
    \[\mathbf{b}_0=\mathbf{J}_{d+1}-\mathbf{I}_{d+1}-\sum_{i=1}^{2m} \mathbf{b}_i.\]
    By Lemma \ref{blocklemma}, the eigenvalues of the adjacency matrix are the union of the eigenvalues of each
    \begin{align*}
        \mathbf{H}_\zeta &= \mathbf{b}_0+\zeta \mathbf{b}_1 + \cdots + \zeta^{2m}\mathbf{b}_{2m}\\
        &= \begin{bmatrix}
            \mathbf{A} & \mathbf{J}_{2m,d-2m+1}\\
            \mathbf{J}_{d-2m+1,2m} & \mathbf{J}_{d-2m+1}-\mathbf{I}_{d-2m+1}
            \end{bmatrix},
    \end{align*}
    where $\zeta$ is a $(2m+1)$-th root of unity and
    \[\mathbf{A}=\mathbf{J}_{2m}+\begin{bmatrix}
        -1 & \overline{\zeta}-1\\
        \zeta-1 & -1\\
        & & -1 & \overline{\zeta^2}-1\\
        & & \zeta^2-1 & -1\\
        & & & &\ddots\\
        & & & & & -1 & \overline{\zeta^m}-1 \\
        & & & & & \zeta^m-1 & -1
    \end{bmatrix}.\]
    The characteristic polynomial $p(x)$ is
    \[p(x)=\prod_{\zeta^{2m+1}=1}\det (x\mathbf{I}-\mathbf{H}_\zeta).\]
    We need the block determinant
    \[\det (x\mathbf{I}-\mathbf{H}_\zeta)=\begin{vmatrix}
            x\mathbf{I}_{2m}-\mathbf{A} & -\mathbf{J}_{2m,d-2m+1}\\
            -\mathbf{J}_{d-2m+1,2m} & (x+1)\mathbf{I}_{d-2m+1}-\mathbf{J}_{d-2m+1}
            \end{vmatrix}.\]
    Label the blocks $\mathbf{A}',\mathbf{B},\mathbf{C},\mathbf{D}$. Passing to the Schur complement,
    \begin{equation}\label{h-xi}\det (x\mathbf{I}-\mathbf{H}_\zeta)=\det(\mathbf{D}) \det(\mathbf{A}'-\mathbf{BD}^{-1}\mathbf{C}).\end{equation} 
    By Lemma \ref{ai+bj},
    \begin{align}\det \mathbf{D}&=(x+1)^{d-2m}(x-d+2m),\label{det D}\\
    \mathbf{D}^{-1}&=\frac{1}{x+1}\mathbf{I}+\frac{1}{(x+1)(x-d+2m)}\mathbf{J}.\nonumber\end{align}
    Note $\mathbf{J}_{n,m}\mathbf{J}_{m,l}=m\mathbf{J}_{n,l}$. Then,
    \begin{align*}
        \mathbf{BD}^{-1}\mathbf{C} &= \frac{d-2m+1}{x+1}\mathbf{J}+\frac{(d-2m+1)^2}{(x+1)(x-d+2m)}\mathbf{J}=\frac{d-2m+1}{x-d+2m}\mathbf{J}, \\
        \mathbf{A}'-\mathbf{BD}^{-1}\mathbf{C}&=-\frac{x+1}{x-d+2m}\mathbf{J}+\begin{bmatrix}
            x+1 & 1-\overline{\zeta}\\
            1-\zeta & x+1\\
            & & x+1 & 1-\overline{\zeta^2}\\
            & & 1-\zeta^2 & x+1\\
            & & & &\ddots\\
            & & & & & x+1 & 1-\overline{\zeta^m}\\
            & & & & & 1-\zeta^m & x+1\\
        \end{bmatrix}.
    \end{align*}
    Let $\mathbf{M}_\zeta$ be the block matrix summand of $\mathbf{A}'-\mathbf{BD}^{-1}\mathbf{C}$ and $\mathbf{u}$ be the all ones vector of dimension $2m$. Notice that $\mathbf{u}\mathbf{u}^T=\mathbf{J}_{2m}$, so that we can apply Lemma \ref{matrix_det_lemma} to give
    \begin{equation}
        \det(\mathbf{A}'-\mathbf{BD}^{-1}\mathbf{C})=\p*{1-\dfrac{x+1}{x-d+2m}\mathbf{u}^T\mathbf{M}_{\zeta}^{-1}\mathbf{u}}\det\mathbf{M}_\zeta\label{det_abdc} .
    \end{equation}
    Here, $\mathbf{u}^T\mathbf{M}_{\zeta}^{-1}\mathbf{u}$ gives the sum of the entries of $\mathbf{M}_{\zeta}^{-1}$. The inverse of a block diagonal matrix is the block diagonal matrix of the inverses of the blocks. Then the $j$th block of $\mathbf{M}_{\zeta}^{-1}$ is
    \[\dfrac{1}{x^2+2x-1+\zeta^j+\overline{\zeta^j}}\begin{bmatrix}
        x+1 & \overline{\zeta^j}-1\\
        \zeta^j-1 & x+1
    \end{bmatrix},\]
    and
    \[\mathbf{u}^T\mathbf{M}_{\zeta}^{-1}\mathbf{u}=\sum_{j=1}^{m} \dfrac{2x+\zeta^j+\overline{\zeta^j}}{x^2+2x-1+\zeta^j+\overline{\zeta^j}}.\]
    The determinant of a block diagonal matrix is the product of the blocks' determinants, so
    \[\det\mathbf{M}_\zeta = \prod_{j=1}^{m} (x^2+2x-1+\zeta^j+\overline{\zeta^j}).\]
    By Lemma \ref{nonprimitive},
    \begin{align*}
       \mathbf{u}^T\mathbf{M}_{\zeta}^{-1}\mathbf{u}&=-\dfrac{1}{2z}+\dfrac{2m+1}{2T_n(z)}(T_{n}(z)-(z-1)U_{n-1}(z)),\\
        \det\mathbf{M}_\zeta&=\dfrac{(2T_n(z))^g}{2z}.
    \end{align*}
    Substituting into \eqref{det_abdc}, we obtain
    \[\det(\mathbf{A'}-\mathbf{BD}^{-1}\mathbf{C})=\dfrac{(2T_{n} (z))^{g-1}}{x-d+2m}\left[\p*{-2m+1+\dfrac{2m-d}{z}} T_{n}(z) + (2m+1)(z-1)U_{n-1}(z)\right].\]
    With \eqref{det D}, we simplify \eqref{h-xi} to
    \[\det(x\mathbf{I}-\mathbf{H}_\zeta)=(x+1)^{d-2m}(2T_{n} (z))^{g-1}\left[ \p*{-2m+1+\dfrac{2m-d}{z}} T_{n}(z) + (2m+1)(z-1)U_{n-1}(z)\right].\]
    Taking the product of $\det(x\mathbf{I}-\mathbf{H}_\zeta)$ over all $2m+1$ roots of unity $\zeta$ yields the characteristic polynomial.
    \end{proof}
    
    \section{Bounding the second eigenvalue} \label{bound}
    Let $[z^n]f(z)$ denote the coefficient of $z^n$ in $f(z)$, and let $z_0\geq z_1\geq \cdots \geq z_{n-1}$ be the roots of $f(z)$. Given any monic polynomial $f(z)$, we can rewrite $f(z)$ as
    $$f(z)=\sum_{j=0}^n a_jz^j = \prod_{j=0}^{n-1}(z-z_j) = z^n - z^{n-1}\sum_{j=0}^{n-1}z_j +O(z^{n-2}).$$

    \begin{lem}[Graeffe's Method \cite{Graeffe2}]\label{Graeffe}
    The largest root of a monic polynomial $f(z)=\sum_{j=0}^n a_jz^j$ can be bounded by  \begin{equation}\label{rootbound}
     z_0 \leq \left(-4a_{n-4}+4a_{n-3}a_{n-1}+2a_{n-2}^2-4a_{n-2}a_{n-1}^2+a_{n-1}^4\right)^{1/4}.
     \end{equation}
    \end{lem}
    
    \begin{proof}
    We compute
    $$f(z)f(-z) = \prod_{j=0}^{n-1}(z-z_j)(z+z_j)=\prod_{j=0}^{n-1}(z^2-z_j^2)=z^{2n}-z^{2n-2}\sum_{j=0}^{n-1} z_j^2+O(z^{2n-4}).$$
    We know that this is an even function, so $a_{2n-1}=0$. Similarly, we can then compute 
    \begin{align*}
    f(z)f(-z)f(\imath z)f(-\imath z) &= \prod_{j=0}^{n-1}(z-z_j)(z+z_j)(z-\imath z_j)(z+\imath z_j)=\prod_{j=0}^{n-1}(z^4-z_j^4) \\
    &=z^{4n}-z^{4n-4}\sum_{j=0}^{n-1} z_j^4+O(z^{4n-8}).
    \end{align*}
    We conclude that
    $$z_0 \leq \left(\sum_{j=0}^{n-1} z_j^4\right)^{1/4} = \left(- [z^{4n-4}]f(z)f(-z)f(\imath z)f(-\imath z)\right)^{1/4}.
     $$
     Moreover, only the powers of $z$ with degree at least $n-4$ multiply to produce a term of degree $4n-4$. By explicitly multiplying these polynomials out, we then have 
     
     \begin{align*}
     f(z)f(-z)&f(\imath z)f(-\imath z)\\&= \prod_{k=1}^4 \left( (\imath ^kz)^n  +a_{n-1}(\imath ^kz)^{n-1} +a_{n-2}(\imath ^kz)^{n-2} +a_{n-3}(\imath ^k z)^{n-3} +a_{n-4}(\imath ^kz)^{n-4} + O(z^{n-5})    \right) \\
     &= z^{4n} - z^{4n-4} \left(-4a_{n-4}+4a_{n-3}a_{n-1}+2a_{n-2}^2-4a_{n-2}a_{n-1}^2+a_{n-1}^4\right) +O(z^{4n-8}).
     \end{align*}
     \end{proof}
     The intuition behind why Graeffe's method suffices here is that if the polynomial's roots are well separated, then Graeffe's method provides extremely precise approximations. Here, numerically we observed $d-1\leq \lambda_2 \leq d$ while the other roots were contained in $[-3,1]$. Then, passing to fourth powers using Graeffe's method, the sum of fourth powers of all the roots is dominated by $\lambda_2^4$, so that taking a fourth root then provides a tight upper bound on the largest eigenvalue. We can use the same technique to derive a series of increasingly strong bounds on the largest root, as we incorporate more of the leading coefficients. For an introduction to the general method, see \cite{Graeffe2}. For historical discussion of the origin of the name and method, see \cite{Graeffe}.
     
     We require the following technical inequality.
    \begin{lem}\label{rootbound}
    For $d \geq 2m+2 \ge 6$, we have
    $$ \p*{d^4+4d^3-(8m-2)d^2+4d+8m^2+4m+1}^{1/4} < d -\frac{2m+1}{d+3}+1.$$
    \end{lem}
    \begin{proof}
    Let $n=2m+1$ and $\ell = d+1$, so this statement reduces to showing
    $$\ell^4-4\ell^2n +2n^2-6n+8\ell n < \left(  \ell -\frac{n}{\ell+2}  \right)^4$$
    for $\ell \geq n+2 \geq 7$. By clearing denominators, expanding, and dividing through by $n$, this is equivalent to showing
     $$ 2\ell^4(2n-5)+8\ell^3(n-6) -4\ell^2(n^2+6n+12) -8\ell (n^2+8n-8) + n^3-32n+96 > 0.$$
    Note that $n^3-32n+96 > 0$ for $n\geq 5$, so we drop this term completely and divide by $2\ell$. It is left to show that
    $$\ell^3(2n-5)+4\ell^2(n-6) +2\ell( -n^2-6n-12) +4 (-n^2-8n+8) \geq 0.$$
    Since we have $\ell - 2 \geq n\geq 5$, we can lower bound the left-hand side as 
    \begin{align*}
    &\geq \ell^3(10-5)+4\ell^2(5-6) +2 \ell\p*{-(\ell-2)^2-6(\ell-2)-12 } +4\p*{-(\ell-2)^2-8(\ell-2) +8} \\
    & = 80 - 24 \ell - 12 \ell^2 + 3 \ell^3.
    \end{align*}
    However, for $\ell \geq n+2 \geq 7$, this polynomial is strictly positive.
    \end{proof}
        
    The main idea of our proof is to explicitly write down the first five coefficients of the characteristic polynomial, apply Graeffe's method, and finally show that Graeffe's method gives a slightly stronger bound than desired.  We can now prove the following.

    \begin{thm} \label{dupperbound}
    For $d\geq 2m+2\geq 4$, $$d-\dfrac{2m+1}{d+1}\leq \lambda_2(\mathcal{G}_{m,d})< d-\dfrac{2m+1}{d+3}.$$
    \end{thm}
    
    \begin{proof}
    See Section \ref{construction} for a discussion of the lower bound. Since $\mathcal{G}_{m,d}$ is a connected $d$-regular graph, it follows that $\lambda_2<\lambda_1=d$ \cite[Ch. 1]{brouwer2011spectra}. In the characteristic polynomial from Theorem \ref{charpoly}, consider when $j=2m+1$ in the product. The expression in the product simplifies to
    \[(x+1)^{2m}(x-d),\]
    from which the eigenvalue of $d$ arises. Also observe that the factors $(x+1)^{(d-2m)(2m+1)}$ and $(2T_n(\frac{x+1}{2}))^{g-1}$ have roots in $[-3,1]$. It is enough to check that the factor
    \[f_n(z)=\p*{-2m+1+\dfrac{2m-d}{z}} T_{n}(z) + (2m+1)(z-1)U_{n-1}(z)\]
    has roots less than $\frac12({d-\frac{2m+1}{d+3}+1})$ for $1\leq j< 2m+1$.
    
    First, let us note some restrictions on the values of $n,m,d$ for which we need to prove this. The case $m=1$ was proven in \cite{wongthesis}. Note that $n = (2m+1)/\gcd(2m+1,j)\ne 1$.  As a divisor of $2m+1$, $n$ must then be odd, which eliminates $n=2$. We are left to check $2m+1 \ge n \ge 3$ and $d \geq 2m+2$ for $m\geq 2$.

    By \eqref{chebytsum} and \eqref{chebyusum}, for $n\geq 4$
    \begin{align*}
        T_n(z)&=2^{n-1}z^n-2^{n-3}nz^{n-2}+2^{n-6}n(n-3)z^{n-4}+O(z^{n-6}),\\
        U_n(z)&=2^nz^n-2^{n-2}(n-1)z^{n-2}+2^{n-5}(n-2)(n-3)z^{n-4}+O(z^{n-6}).
    \end{align*}
    Then for $n\geq 4$,
    \begin{align*}
    f_n(z)&= \p*{-2m+1+\dfrac{2m-d}{z}} T_{n}(z) + (2m+1)(z-1)U_{n-1}(z) \\
    &= \p*{-2m+1+\dfrac{2m-d}{z}} (2^{n-1}z^n-2^{n-3}nz^{n-2}+2^{n-6}n(n-3)z^{n-4}+O(z^{n-6}))\\
    &\quad +(2m+1)(z-1)(2^{n-1}z^{n-1}-2^{n-3}(n-2)z^{n-3}+2^{n-6}(n-3)(n-4)z^{n-5}+O(z^{n-6}))
    \end{align*}
    so that the first five coefficients of the characteristic polynomial are explicitly
    \begin{align}
    \dfrac{f_n(z)}{2^n}    &=z^n-\dfrac{d+1}{2}z^{n-1}+\dfrac{2m+1-n}{4}z^{n-2}+\dfrac{dn+n-4m-2}{8}z^{n-3} \label{fullpoly}\\
        &\quad +\dfrac{(n-4m-2)(n-3)}{32}z^{n-4}+O(z^{n-5}).  \nonumber
    \end{align}

    We examine the behavior at $n=3$, where \eqref{fullpoly} is technically undefined. Setting $n=3$ there, the $z^{n-4}$ factor vanishes due to the $(n-3)$ term, and we recover the correct polynomial. Therefore, we can extend the validity of this statement to $n\geq 3$.

    By applying Graeffe's method (Lemma \ref{Graeffe}) with
    \begin{align*}
        a_{n-1} &= -\dfrac{d+1}{2},\\ a_{n-2}&=\dfrac{2m+1-n}{4},\\ a_{n-3} &= \dfrac{dn+n-4m-2}{8}, \\a_{n-4} &= \dfrac{(n-4m-2)(n-3)}{32},
    \end{align*}
    the largest root $z_0$ of $f_n(z)$ satisfies
    \begin{equation}\label{root1}
    z_0\leq \dfrac{1}{2}\p*{d^4+4d^3-(8m-2)d^2+4d+8m^2-8m+6n-5}^{1/4}.
    \end{equation}
    Note that for fixed $m$ and $d$, this function is increasing in $n$ due to the $6n$ term, so it suffices to take $n=2m+1$ (the maximum possible value of $n$) in \eqref{root1}. Then \eqref{root1} reduces to
    \begin{equation}\label{root2}
    z_0\leq \dfrac{1}{2}\p*{d^4+4d^3-(8m-2)d^2+4d+8m^2+4m+1}^{1/4}.
    \end{equation}
    Lemma \ref{rootbound} shows that this is less than $\frac12({d-\frac{2m+1}{d+3}+1})$ for $d\geq 2m+2 \ge 6$. 
    \end{proof}
    
    \section{Application to Rigidity}\label{rigidity}
    One interesting application of this graph family pertains to graph rigidity. Rigidity is a well-studied notion of resistance to bending. A \textit{$k$-dimensional} framework is a pair $(G,p)$, where $G$ is a graph and $p$ is a map from $V(G)$ to $\mathbb{R}^k$. Let $||\cdot||$ denote the Euclidean norm in $\mathbb{R}^k$. Two frameworks $(G,p)$ and $(G,q)$ are \textit{equivalent} if $||p(u) - p(v)|| = ||q(u) - q(v)||$ for every edge $uv \in E(G)$, and \textit{congruent} if $||p(u) - p(v)|| = ||q(u) - q(v)||$ for every $u, v \in V(G)$. Further, a framework $(G,p)$ is \textit{generic} if the coordinates of its points are algebraically independent over the rationals. Finally, a framework $(G,p)$ is \textit{rigid} if there exists an $\varepsilon > 0$ such that if $(G,p)$ is equivalent to $(G,q)$ and $||p(u)-q(u)|| < \varepsilon$ for every $u \in V(G)$, then $(G,p)$ is congruent to $(G,q)$. Note that we can consider rigidity as a property of the underlying graph, as a generic realization of $G$ is rigid in $\mathbb{R}^k$ if and only if every generic realization of $G$ is rigid in $\mathbb{R}^k$ \cite{asimow1978rigidity}. Thus, we call a graph \textit{rigid} in $\mathbb{R}^k$ if every generic realization of $G$ is rigid in $\mathbb{R}^k$. For the remainder of the section, we only consider rigid graphs in $\mathbb{R}^2.$
    
    Cioab\u{a}, Dewar, and Gu investigated the connection between a graph's Laplacian eigenvalues and rigidity, ultimately arriving at the following sufficient spectral condition.
    \begin{thm}[\cite{cioabua2020spectral}]\label{rigidcriteria}
    Let $G$ be a graph with minimum degree $\delta(G) \ge 6r$. If
    \begin{enumerate}
        \item $\mu_2(G) > \frac{6r - 1}{\delta(G) + 1}$,
        \item $\mu_2(G - u) > \frac{4r-1}{\delta(G-u) + 1}$ for every $u \in V(G)$, and 
        \item $\mu_2(G - v - w) > \frac{2r-1}{\delta(G - v - w) + 1}$ for every $v, w \in V(G)$,
    \end{enumerate}
    then $G$ contains at least $r$ edge-disjoint spanning rigid subgraphs.
    \end{thm}
    Moreover, they used a variant of $\mathcal{G}_{2,d}$ and a result by Lov\'{a}sz and Yemini \cite{lovaszrigid} to show that the first condition of the above theorem essentially cannot be improved for $r = 1$. However, this categorization does not extend beyond $r=1$; we must turn to a more recent result by Gu \cite{gu2018spanningrigit} to verify tightness for larger values.
    
    Suppose $\pi$ is a partition of $V(G)$. We call a part of $\pi$ \textit{trivial} if it consists of a single vertex. Additionally, let $e_G(\pi)$ denote the number of edges of $G$ whose endpoints lie in two different parts of $\pi$.
    
    \begin{thm}[\cite{gu2018spanningrigit}]\label{necessaryrigitcondition}
    Let $r, \ell \ge 0$ be integers. If a graph $G$ contains $r$ spanning rigid subgraphs and $\ell$ spanning trees, all of which are mutually edge-disjoint, then for any partition $\pi$ of $V(G)$ with $t$ trivial parts, $$e_G(\pi) \ge (3r+\ell)(|\pi|-1) - rt.$$
    \end{thm}
    
    Using this result, we are able to prove the first condition of Theorem \ref{rigidcriteria} is essentially tight for all $r \ge 1$.
    
    \begin{thm}
    Let $r > 0$ be an integer. Then $\mathcal{G}_{3r-1,d}$ has fewer than $r$ edge-disjoint spanning rigid subgraphs and satisfies
    \[
    \frac{6r-1}{d+3} < \mu_2(\mathcal{G}_{3r-1,d}) \le \frac{6r-1}{d+1}.
    \]
    \end{thm}
    
    \begin{proof}
    Let $\pi$ be a partition of $V(\mathcal{G}_{3r-1, d})$, where the vertices of each modified $K_{d+1}$ form a part. Then $|\pi| = 2(3r-1)+1 = 6r-1$ and $e_{\mathcal{G}_{3r-1, d}}(\pi) = (3r-1)(2(3r-1)+1) = (3r-1)(6r-1)$. Thus,
    \[
    e_{\mathcal{G}_{3r-1, d}}(\pi) = (3r-1)(6r-1) < 3r(6r-2) \le (3r+\ell)((6r - 1)-1)
    \]
    for $\ell \ge 0$, so $\mathcal{G}_{3r-1, d}$ cannot contain $r$ edge-disjoint spanning rigid subgraphs by Theorem \ref{necessaryrigitcondition}.
    
    The spectral inequality is an immediate consequence of Theorem \ref{extremegraph} by recognizing that $\mu_2 = d - \lambda_2$, since the graph is $d$-regular.
    \end{proof}
    
    \section{Open Problem}
    Numerically, $\lambda_2(\mathcal{G}_{m,d})$ converges to the upper bound $d-\frac{2m+1}{d+3}$ from Theorem \ref{extremegraph}. It is natural to ask the following: what is the optimal function $g(m,d)$ so that if a $d$-regular graph $\mathcal{G}$ has $\lambda_2(\mathcal{G})\leq g(m,d)$, then $G$ contains at least $m+1$ edge-disjoint spanning trees, and what are the extremal graphs?
    
    
    \section{Acknowledgements}
    The researach of Sebastian M. Cioab\u{a} was partially supported by NSF grants DMS-1600768 and CIF-1815922.
    Tanay Wakhare was supported by the MIT Television and Signal Processing Fellowship.

    \bibliographystyle{plain}
    \bibliography{bibliography}
    \end{document}